\documentclass[onecolumn]{svjour3}       
\smartqed  
\usepackage[utf8]{inputenc}
\usepackage{graphicx}
\usepackage{amssymb}

\title{Extremal $K_4$-minor-free graphs without short cycles\thanks{Supported by ERC Advanced Grant "GeoScape", National Research, Development and Innovation Office, NKFIH,
K-131529. and the grant of the 
Hungarian Ministry for Innovation and Technology (Grant Number: NKFIH-1158-6/2019).}}
\author{János Barát}
\institute{János Barát \at
Alfréd Rényi Institute of Mathematics, Eötvös Loránd Research Network, Budapest, Hungary\\
 and
Department of Mathematics, University of Pannonia, Hungary\\
\email{barat@mik.uni-pannon.hu}   }

\date{Received: date / Accepted: date}

\begin{document}

\maketitle

\begin{abstract}
We determine the maximum number of edges in a $K_4$-minor-free $n$-vertex graph of girth $g$, when $g=5$ or $g$ is even.
We argue that there are many different $n$-vertex extremal graphs, if $n$ is even and $g$ is odd.  
\end{abstract}

\section{Introduction}

We assume the reader is familiar with basic graph theory, see e.g. \cite{diestel}.
In particular, let $n$  denote the number of vertices,  $e$ the number of edges, $f$ the number of faces of a planar graph.
A face with precisely $i$ edges on its boundary is an $i$-face.
The {\em girth} of a graph is the length of the shortest cycle.  
We denote the distance\footnote{length of the shortest path} between vertices $u$ and $v$ in a graph $H$ by $d_H(u,v)$.  

Let $u$ and $v$ be adjacent vertices. 
The {\em contraction} of an edge $uv$ is the following operation.
We delete the edge $uv$ and merge the vertices $u$ and $v$ into a new vertex $w$ such that 
edges incident to either $u$ or $v$ correspond to edges incident to $w$. 
For our purposes, we only keep single edges if multiple edges are created. 
A graph $H$ is a {\em minor} of a graph $G$ if we can get $H$ from $G$ by contracting and deleting edges.
A graph $G$ contains a complete graph $K_t$ as a minor if $G$ contains $t$ disjoint connected subgraphs such that
any two such subgraphs are connected by an edge of $G$.
A graph $G$ is {\em $K_t$-minor-free} if $G$ does not contain $K_t$ as a minor.

Let $C$ be a cycle of a graph $G$, and let $F$ be a connected component of $G-C$.
The edges between vertices of $F$ and $C$ are the {\em leg}s of $F$.
A {\em bridge} $B$ of $C$ in $G$ is a connected component of $G-C$ together with its legs.
The vertices in $B\cap C$ are the {\em attachment} vertices.

In an {\em extremal} problem, we would like to determine the maximum number of edges of an $n$-vertex graph $G$
that belongs to a graph class ${\mathcal C}$.
The graphs that have this maximum number of edges are the {\em extremal graph}s.
In the classical example, Turán's theorem, there is precisely one extremal graph.
In our focus, the graph class is the $K_t$-minor-free graphs of girth $g$, where both $t$ and $g$ are small.
In some cases there is only one extremal graph and in some other cases there are many.

Our motivation comes from graph colorings.
An easy list-coloring argument is the following. Assume every vertex has a list of colors of size $k$.
Consider a class $\mathcal{C}$ of graphs, where every graph $G$ in the class has minimum degree $\delta(G)\le k-1$.
Now every graph in the class can be $k$-list-colored by the greedy algorithm, always coloring a vertex 
of minimum degree.

Our distant goal is to study $K_6$-minor-free graphs without short cycles.
The preliminary results show that proving an extremal result might be enough to use the simple algorithm given above to 
improve on the present upper bounds.
To find the necessary tools for the problem, we first study $K_4$- and $K_5$-minor-free graphs with a 
condition on the girth. See also the separate paper on $K_5$-minor-free graphs \cite{bj1}.

Any $K_4$-minor-free graph is a subgraph of a $3$-tree. 
That is, a graph built from triangles pasted along edges.
The $K_4$-minor-free graphs are also called {\em series-parallel} graphs.
Observe that any $K_4$-minor-free graph is planar. 



\section{Even girth}

Let $G$ be a $K_4$-minor-free graph and $C$ a cycle of $G$.
Studying planar graphs, the relative position of bridges of a cycle plays a crucial role \cite{b&m}.
Two bridges $B_1$ and $B_2$ of $C$ are {\em crossing}, if there are attachment vertices $u_1,u_2$ of $B_1$ and
$v_1,v_2$ of $B_2$ such that $u_1,v_1,u_2,v_2$ appears on $C$ in this order. 
We use the following simple facts. 

\begin{proposition}
 Let $C$ be any cycle of a $K_4$-minor-free graph $G$.\\ 
 (i) Every bridge of $C$ can have at most $2$ legs.\\
 (ii) The bridges of $C$ are non-crossing.
\end{proposition}

Since $G$ is $K_4$-minor-free, it has a planar drawing.
Consider $G^*$ the dual of any drawing of $G$, where we ignore the unbounded region.
Any cycle in $G^*$ would imply a $K_4$-minor in $G$.
Therefore, $G^*$ must be a tree.

Let $G$ be a $K_4$-minor-free graph of girth $2k$, where $k\ge 2$.

\begin{proposition}
 A $K_4$-minor-free graph $G$ on $n$ vertices of girth $2k$ can have at most $\frac{k}{k-1}(n-2)$ edges, where $k\ge 2$. 
 This bound is tight, if $n=s(k-1)+2$, where $s\ge 2$.
\end{proposition}

\begin{proof}
 The graph $G$ satisfies Euler's formula: $e+2=f+n$.
 Also, by the girth condition $2e=\sum i f_i\ge 2kf$, where $f_i$ denotes the number of $i$-faces of $G$.
 Therefore, $2ke+4k=2kf+2kn\le 2e+2kn$, which implies $e\le \frac{2kn-4k}{2k-2}=\frac{k}{k-1}(n-2)$.
 
 Next, we construct a $K_4$-minor-free graph $G$ with $n=s(k-1)+2$ vertices and $e=ks$ edges and of girth $2k$,
 where $s\ge 2$.
 This gives us a sharp example as $\displaystyle\frac{k\left(s(k-1)+2\right)-2k}{k-1}=ks$.
 The graph $G$ is the following: let $u$ and $v$ be two vertices and connect them by $s$ edge-disjoint paths 
 with $k$ edges.
 Now every cycle must go through $u$ and $v$, hence the girth is indeed $2k$.
 The graph $G$ is also $K_4$-minor-free since removing any cycle leaves only bridges with 2 legs.
 
 We show that actually these are all the extremal examples, when $n=s(k-1)+2$.
 Notice that the application of Euler's formula shows us that in any drawing of an extremal graph $G$
 all faces must have $2k$ edges.
 Let $C$ be a face of length $2k$ in $G$.
 Consider all bridges of $C$ in $G$.
 Any such bridge $B$ can have at most two legs (actually must have 2 legs).
 Also bridges must be non-crossing.
 Finally, any two bridges must have the same attachment vertices.
 Otherwise we get a face with length greater than $2k$.
 Now we know that there are two vertices $u$ and $v$ of $C$ and all bridges are attached to these two vertices in $G$.
 
 Taking into account that every face has length $2k$, we get that a bridge $B$ cannot contain a path longer than the path
 between $u$ and $v$ in $C$. This implies that indeed $B$ cannot be other than a copy of the path between $u$ and $v$.
$\Box$ \end{proof}

\section{Girth 5}

In this section $G$ denotes a $K_4$-minor-free graph of girth 5.
We start with an important observation.

\begin{lemma} \label{l2}
 We may assume that an extremal graph $G$ is $2$-connected.
\end{lemma}

\begin{proof}
 Suppose that a $K_4$-minor-free graph $G'$ of girth $5$ has a cutvertex $x$ such that
 $G_1\cup G_2=G'$ and $G_1\cap G_2=x$.
 Let $v_1$ be a neighbor of $x$ in $G_1$ and $v_2$ a neighbor in $G_2$.
 We modify the graph $G'$ by deleting the edge $v_2x$ and adding the edge $v_1v_2$.
 Notice the girth condition is still satisfied.
 How could there be a $K_4$-minor now? 
 Only if there is a cycle $C$ through $x$ and $v_1v_2$ and a bridge $B$ of $C$ with at least three legs. 
 However, the attachment vertices of $B$ must be entirely in $G_1$ or $G_2$.
 If they are in $G_1$, then we contract the path of $C$ between $x$ and $v_1$ through $v_2$ to find a $K_4$-minor 
 already in $G_1$.
 If the attachment vertices are in $G_2$, then we contract the path of $C$ between $x$ and $v_2$ through $v_1$ 
 to find a $K_4$-minor already in $G_2$.
 It shows that the modified graph is still $K_4$-minor-free.
 Repeating the above process we can transform the original graph to one without cutvertices, but still having the
 same parameters.  
$\Box$ \end{proof}

In what follows, we assume $G$ is $2$-connected.
As it often happens, the odd girth case is different and more complicated than even girth.
We can use Euler's formula as before to set an upper bound. 
However, there is no matching lower bound construction.


\begin{theorem}\label{main}
 Any $n$-vertex $K_4$-minor-free graph $G$ of girth $5$ can have at most $\lceil\frac{3}{2}n-3\rceil$ edges, where $n\ge 5$.
\end{theorem}

\begin{proof}
 We prove the statement by induction on the number of vertices $n$.
 For the base of induction, we notice that on 5 vertices, $G$ can have at most $5$ edges.
 We get $5\le \lceil\frac{3}{2}5-3\rceil=5$ and the statement holds.
 On 6 vertices, the 6-cycle is the only possible graph and the claim holds again.
 Finally on 7 vertices, there are two possible graphs and one extremal example.
 The inequality holds for these too.
 
 For the induction step, let $G$ be a $2$-connected extremal $K_4$-minor-free graph of girth $5$.
 Let $\{u,v\}$ be a $2$-cut\footnote{There is always a $2$-cut since $G$ is the subgraph of a $3$-tree.} 
 such that $G_1\cup G_2=G$ and $G_1\cap G_2=\{u,v\}$.
 Let $n_i$ and $m_i$ denote the number of vertices in $G_i$ for $i\in\{1,2\}$.
 Observe here that we can always select $u$ and $v$ such that $G_1$ has at least 4 vertices\footnote{That is, the graph $G_1$ is 
 larger than a path with 2 edges.}, when $n\ge 8$.
 
 First suppose that $d_{G_1}(u,v)=1$ and $d_{G_2}(u,v)\ge 4$.
 We add an edge between $u$ and $v$ in $G_2$ to create $G_2^+$.
 By the induction hypothesis for $G_1$ and $G_2^+$, we deduce the following:\\
 $m_1\le \lceil\frac{3}{2}n_1-3\rceil\le\frac{3}{2}n_1-\frac{5}{2}$.\\
 $m_2+1\le \lceil\frac{3}{2}n_2-3\rceil\le\frac{3}{2}n_2-\frac{5}{2}$.
  
 Since $m=m_1+m_2$ and $n=n_1+n_2-2$, we get:\\
 $m=m_1+m_2\le \frac{3}{2}n_1-\frac{5}{2}+\frac{3}{2}n_2-\frac{7}{2}=\frac{3}{2}(n_1+n_2)-6=\frac{3}{2}n-3\le
 \lceil\frac{3}{2}n-3\rceil$.

 \medskip
 
 Now suppose that $d_{G_1}(u,v)\ge 2$ and $d_{G_2}(u,v)\ge 3$.
 
 \underline{Case 1.} $n_1$ and $n_2$ are even.
 Now we directly apply the induction hypothesis on  $G_1$ and $G_2$:\\
 $m_1\le \lceil\frac{3}{2}n_1-3\rceil=\frac{3}{2}n_1-3$.
 
 \smallskip
 
 \noindent $m_2\le \lceil\frac{3}{2}n_2-3\rceil=\frac{3}{2}n_2-3$.
 
 Since $m=m_1+m_2$ and $n=n_1+n_2-2$, we get:\\
 $m=m_1+m_2\le \frac{3}{2}n_1-3+\frac{3}{2}n_2-3=\frac{3}{2}n-3=\lceil\frac{3}{2}n-3\rceil$, since $n$ is even.

 \underline{Case 2.} $n_1$ and $n_2$ are odd.
 Now we add a path with 2 edges between $u$ and $v$ in $G_2$ to create $G_2^+$.
 By the induction hypothesis for $G_1$ and $G_2^+$, we deduce the following:\\
 $m_1\le \lceil\frac{3}{2}n_1-3\rceil=\frac{3}{2}n_1-\frac{5}{2}$.\\
 $m_2+2\le \lceil\frac{3}{2}(n_2+1)-3\rceil=\frac{3}{2}(n_2+1)-3$.
 
 Since $m=m_1+m_2$ and $n=n_1+n_2-2$, we get:\\
 $m=m_1+m_2\le \frac{3}{2}n_1-\frac{5}{2}+\frac{3}{2}(n_2+1)-5=\frac{3}{2}(n_1+n_2)-6=\frac{3}{2}n-3\le
 \lceil\frac{3}{2}n-3\rceil$.
 
 \underline{Case 3.} $n_1$ is even and $n_2$ is odd.
 Again we add a path with 2 edges between $u$ and $v$ in $G_2$ to create $G_2^+$.
 By the induction hypothesis for $G_1$ and $G_2^+$, we deduce the following:\\
 $m_1\le \lceil\frac{3}{2}n_1-3\rceil=\frac{3}{2}n_1-3$.\\
 $m_2+2\le \lceil\frac{3}{2}(n_2+1)-3\rceil=\frac{3}{2}(n_2+1)-3$.
 
 Since $m=m_1+m_2$ and $n=n_1+n_2-2$, we get:\\
 $m=m_1+m_2\le \frac{3}{2}n_1-3+\frac{3}{2}(n_2+1)-5=\frac{3}{2}(n_1+n_2)-8+\frac{3}{2}=\frac{3}{2}n-\frac{7}{2}<
 \frac{3}{2}n-3
 \le\lceil\frac{3}{2}n-3\rceil$.
 
 \underline{Case 4.} $n_1$ is odd and $n_2$ is even.
 We add a path with 2 edges between $u$ and $v$ in $G_2$ to create $G_2^+$.
 By the induction hypothesis for $G_1$ and $G_2^+$, we deduce the following:\\
 $m_1\le \lceil\frac{3}{2}n_1-3\rceil=\frac{3}{2}n_1-\frac{5}{2}$.\\
 $m_2+2\le \lceil\frac{3}{2}(n_2+1)-3\rceil=\frac{3}{2}(n_2+1)-\frac{5}{2}$.
 
 Since $m=m_1+m_2$ and $n=n_1+n_2-2$, we get:\\
 $m=m_1+m_2\le \frac{3}{2}n_1-\frac{5}{2}+\frac{3}{2}(n_2+1)-\frac{5}{2}-2=\frac{3}{2}(n_1+n_2)-7+\frac{3}{2}=
 \frac{3}{2}n-4+\frac{3}{2}=\frac{3}{2}n-\frac{5}{2}=
 \lceil\frac{3}{2}n-3\rceil$, since $n$ is odd.
$\Box$ \end{proof}

\section{Lower bound construction}

Next, we show a $K_4$-minor-free graph $G_s$ with $n=2s+1$ vertices and $e=3s-1=\lceil\frac{3}{2}n-3\rceil$ edges and of girth $5$ for
every $s\ge 2$. 
This shows that Theorem~\ref{main} is tight for every odd $n$, where $n\ge 5$.
The construction is similar to the one for even girth.
Let $u$ and $v$ be two vertices and connect them by $s-1$ edge-disjoint paths 
 with $3$ edges and one path with $2$ edges. See Fig~\ref{constr1}.
 
\begin{figure}[ht]
  \begin{center}
    \includegraphics[width=0.2\textwidth]{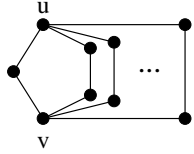}
    \caption{\label{constr1}Extremal $K_4$-minor-free graphs of girth $5$.}
  \end{center}
\end{figure}

If we subdivide an edge\footnote{replace an edge by a $2$-edge path} of $G_{s-1}$ 
not in the $2$-edge path between $u$ and $v$, then we get a $K_4$-minor-free graph of girth 5 on $n=2s$ vertices and 
$e=3s-3=\frac{3}{2}n-3$ edges.
Next, we indicate that if the number of vertices is even, there are many other graphs with the same parameters.

We determine all $K_4$-minor-free graphs on 10 vertices of girth 5 and 12 edges.
Let us consider the planar drawings of these graphs.
We distinguish them by the length $l$ of their outer cycle $C$.

If $l=10$, then there must be two more edges connecting vertices of $C$.
Since $G$ is $K_4$-minor-free, these two edges form non-crossing bridges.
So there are 3 faces inside of $C$.
Now two edges can make the total length of these 3 faces to $14$.
There must be a face with at most 4 edges, a contradiction.

If $l=9$, then there must be a vertex $x$ in $C$ such that $x$ is adjacent to $y$ in $C$ forming a $5$-face and 
$x$ is connected to $z$ in $C$ by a subdivided edge forming two more $5$-faces.
This gives us the first example: $H_1$, see Fig~\ref{ext10}.

If $l=8$, then the remaining 2 vertices and 4 edges can form an edge and a 3-edge-path or two 2-edge-paths
connecting vertices of $C$.
The total length of the faces inside $C$ must be 16.
If an edge is connecting two vertices $x$ and $y$ of $C$, it must form two 5-cycles with $C$.
We may assume that the remaining 3-path is also incident to $x$.
There are two different ways of connecting $x$ to $z$ in $C$
by a $3$-path. 
In the drawing, either the length of the faces are 5,6,5 or 6,5,5.
Now we have to check whether these examples are isomorhic to $H_1$.
We find that indeed the 5,6,5 example is $H_1$, so we found a new graph $H_2$, see Fig~\ref{ext10}.

If there are two $2$-edge-paths, then they either have a vertex $x$ in common or not.
If they are independent, then they form faces of length 5,6,5 inside $C$.
This is a new graph $H_3$ by its degree sequence.
If $x$ is connected to $y$ and $z$ in $C$ by two $2$-paths, then the faces have length either 
5,6,5 or 5,5,6 in order. These are also new examples: $H_4$ and $H_5$, see Fig~\ref{ext10}.

If $l=7$, then there are 3 remaining vertices and 5 edges, so the total length of the faces inside $C$ must be 17.
Any chord of $C$ would create a short cycle and a 4-path would create a long cycle tracking us back to a previous case. 
So there can be a 2-path $Q$ and a 3-path $P$ connecting vertices of $C$.
The 2-path must connect two vertices at distance 3 on $C$ creating a 5-cycle and a 6-cycle.
If the 3-path $P$ has attachments different from the 2-path, then $P$ connects two vertices at distance 2 on the 
created 6-cycle. However, this is isomorphic to $H_3$.
Otherwise a common vertex $x$ is incident to both $Q$ and $P$.
Principally there are four possibilities for the other endvertex of $P$.

1. The other endvertex is on the 5-cycle at distance 2 from $x$.
This creates faces of length 5,6,6 inside $C$. 
However, this is isomorphic to $H_5$.

2. The other endvertex coincides with the other endvertex of $Q$.
Now there are two vertices of degree 4, and the length of the faces inside $C$ are 5,5,7. 
This is a new example $H_6$, see Fig~\ref{ext10}.

3. The other endvertex of $P$ is at distance 3 from $x$ on $C$ and is different from the other endvertex of $Q$.
This creates faces of length 5,6,6 inside $C$.
This is again a new example $H_7$, see Fig~\ref{ext10}.

4. The other endvertex of $P$ is at distance 2 from $x$ on $C$.
This creates faces of length 5,7,5 inside $C$.
However, this example is isomorphic to $H_4$.

If $l=6$, then there are 4 remaining vertices and 6 edges. 
If they form a path with at least 4 edges, then a cycle of length at least 7 is created.
That was covered in the previous cases.
Therefore the remaining bridges must be two 3-paths.
If we connect two vertices at distance 2 on $C$ by a 3-path, then we create a cycle of length 7.
Therefore the 3-paths must connect vertices at distance 3. 
Since the bridges are non-crossing, the attachments must be the same for the two 3-paths.
This is graph $H_8$, see Fig~\ref{ext10}.

\begin{figure}[ht]
  \begin{center}
    \includegraphics[width=0.7\textwidth]{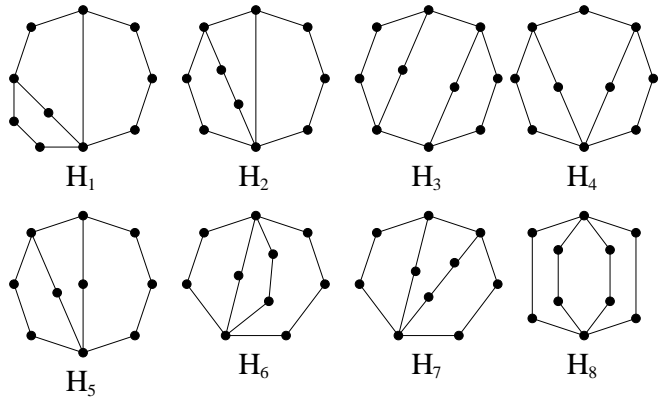}
    \caption{\label{ext10}Extremal $K_4$-minor-free graphs of girth $5$ on $10$ vertices.}
  \end{center}
\end{figure}

We emphasize again, that thereby we proved that Theorem~\ref{main} is tight for $n=10$, 
and we have determined the complete list of extremal graphs as shown in Fig~\ref{ext10}.

\begin{acknowledgements}
 We thank an anonymous referee of \cite{bj1}, who suggested this project.
 We thank another anonymous referee for suggesting structural changes that improved the presentation.
\end{acknowledgements}

\end{document}